\documentclass[a4paper,reqno,final]{amsart}
\usepackage{enumerate}
\usepackage{bm}
\usepackage[foot]{amsaddr}
\usepackage{amsthm}
\usepackage{nicefrac}
\usepackage{cite}
\usepackage{color}
\usepackage[hidelinks]{hyperref}
\usepackage{graphicx}
\usepackage{mathrsfs}

\usepackage{a4wide}

\newtheorem{theorem}{Theorem}
\newtheorem{lemma}{Lemma}
\newtheorem{proposition}{Proposition}
\newtheorem{definition}{Definition}
\theoremstyle{remark}
\newtheorem{remark}{Remark}
\newtheorem{apt}{Assumption}
\renewcommand{\L}[1]{\mathrm{L}^{#1}(\Omega)}
\newcommand{\Ltwo}{\L{2}}

\renewcommand{\H}{\mathrm{H}}
\newcommand{\B}{\mathrm{B}}
\newcommand{\R}{\mathbb{R}}
\newcommand{\N}{\mathbb{N}}

\newcommand{\F}{\mathsf{F}}

\newcommand{\W}{\mathbb{W}_N}
\newcommand{\WW}{\mathbb{W}}
\newcommand{\Wp}{\mathbb{W}_{N(p)}}
\newcommand{\cT}{\mathcal{T}}
\renewcommand{\u}{u}
\newcommand{\uW}{U}

\newcommand{\U}{U}
\newcommand{\dd}{\mathsf{d}}

\newcommand{\dx}{\dd\mat x}
\newcommand{\res}{\eta}
\newcommand{\Res}{\xi}

\newcommand{\eps}{\varepsilon}
\newcommand{\errex}{\mathfrak{e}_{\W}}
\newcommand{\NN}[1]{\|\nabla#1\|}

\newcommand{\cD}{\mathcal{D}}
\newcommand{\cN}{\mathcal{N}}
\newcommand{\HGamma}{\H^1_{\cD}(\Omega)}
\renewcommand{\O}{\mathcal{O}}
\newcommand{\set}[1]{\mathcal{#1}}
\newcommand{\mat}[1]{\bm{\mathsf{{#1}}}}
\newcommand{\dprod}[2]{\left<#1,#2\right>}
\newcommand{\Sp}[1]{\mathbb{S}^p_{\cD}(\Omega;#1)}
\newcommand{\Spk}[1]{\mathbb{S}^{p_k}_{\cD}(\Omega;#1)}

\newcommand{\tp}{\mathfrak{p}}
\newcommand{\Pf}{\Pi_{\tp}(f)}
\newcommand{\tU}{\mathscr{U}}
\DeclareMathOperator{\D}{D}
\DeclareMathOperator{\dist}{dist}
\DeclareMathOperator{\diam}{diam}
\DeclareMathOperator{\work}{work}
\DeclareMathOperator{\Div}{div}

\title[%
       $hp$-ILGFEM
       for semilinear elliptic PDE]{%
       Exponential Convergence of $hp$-ILGFEM
       for semilinear elliptic boundary value problems with monomial reaction}

\author[Y.~He]{Yanchen He}
\address{Seminar for Applied Mathematics, ETH Z\"urich, CH-8092 Z\"urich, Switzerland}

\author[P.~Houston]{Paul Houston}
\address{School of Mathematical Sciences, University of Nottingham, University Park,
Nottingham NG7 2RD, UK}

\author[C.~Schwab]{Christoph Schwab}
\address{Seminar for Applied Mathematics, ETH Z\"urich, CH-8092 Z\"urich, Switzerland}

\author[T.~P.~Wihler]{Thomas P.~Wihler}
\address{Mathematics Institute, University of Bern, CH-3012 Bern, Switzerland}
\thanks{TPW acknowledges the support of the Swiss National Science Foundation, Grant No. 200021\_212868}

\date{Draft---\today}

\begin{document}

\begin{abstract}
We study the fully explicit numerical approximation of a semilinear elliptic boundary value model problem, which 
features a monomial reaction and analytic forcing, 
in a bounded polygon $\Omega\subset\R^2$ with a finite number of straight edges. 
In particular, we analyze
the convergence of $hp$-type iterative linearized Galerkin ($hp$-ILG) solvers.
Our convergence analysis is carried out for conforming $hp$-finite element (FE) 
Galerkin discretizations
on sequences of regular, simplicial partitions of $\Omega$, with
geometric corner refinement, with polynomial
degrees increasing in sync with the geometric mesh refinement
towards the corners of $\Omega$.
For a sequence of 
discrete solutions generated by the ILG solver, with a
stopping criterion that is consistent with the exponential convergence 
of the exact $hp$-FE Galerkin solution, we prove exponential convergence in $\H^1(\Omega)$
to the unique weak solution of the boundary value problem. 
Numerical experiments illustrate the exponential convergence 
of the numerical approximations obtained from the proposed scheme 
in terms of the number of degrees of freedom 
as well as of the computational complexity involved.
\end{abstract}
 
\maketitle

\section{Introduction}
%
On an open bounded polygon $\Omega\subset\R^2$ with straight edges we aim to design and analyze a  fully discrete numerical approximation scheme in terms of an exponentially convergent $hp$-version finite element method ($hp$-FEM) for the semilinear elliptic diffusion-reaction model equation
\begin{subequations}\label{eq:sl}
\begin{align}\label{eq:sla}
-\Delta u+\lambda u^{2q+1}&=f\qquad\text{in }\Omega,
\intertext{%
where $q\in\mathbb{N}_0$ is a fixed integer, $\lambda> 0$ is a positive constant, and $f\in\L{\nicefrac{2(q+1)}{(2q+1)}}$ is a given source function (independent of $u$).
In order to impose \emph{Dirichlet} and \emph{Neumann} boundary conditions, we signify the individual (straight) edges of the boundary $\Gamma:=\partial\Omega$ by $\Gamma_i$, $i=1,2,\ldots,m$; here, we suppose that each edge $\Gamma_i$ is the line connecting two consecutive corners $\mat c_i$ and $\mat c_{i+1}$ of~$\Omega$ (where we let $\mat c_{m+1}=\mat c_1$). For a \emph{nonempty} index set $\cD\subseteq\{1,2,\dots,m\}$, which will be used to identity \emph{Dirichlet} boundary edges, we define the remaining indices $\cN:=\{1,2,\dots,m\}\setminus\cD$ to specify the \emph{Neumann} boundary part. In accordance with the (disjoint) decomposition $\cD\cup\cN=\{1,\ldots,m\}$, we introduce the boundary conditions
}
u&=0\qquad\text{on }\Gamma_{\cD}:=\bigcup_{i\in\cD}\Gamma_i,\label{eq:Dirichlet}
\intertext{and}
\partial_{\mat n} u&=0\qquad\text{on }\Gamma_{\cN}:=\bigcup_{i\in\cN}\Gamma_i,
\end{align}
\end{subequations}
where $\partial_{\mat n} u$ is the outward normal derivative along~$\Gamma_{\cN}$.

Admitting a variational formulation, see~\eqref{eq:weak} below, the boundary value problem~\eqref{eq:sl} will be naturally discretized by applying conforming Galerkin projections on a sequence of subspaces $\W \subset \HGamma$ of finite dimension $N \sim \dim(\W)$, where
\[
\HGamma:=\{v\in \H^1(\Omega):v|_{\Gamma_{\cD}}=0\}
\] 
is the usual Sobolev space of all weakly differentiable functions in the Lebesgue space~$\Ltwo$, with first-order partial derivatives in $\Ltwo$, and zero boundary values along $\Gamma_\cD$ in the sense of traces. Owing to the monotone structure of the nonlinear reaction term in \eqref{eq:sl}, the associated standard weak formulation is uniquely solvable, both in the continuous and in the discrete setting, and a quasi-optimality property can be derived (see Propositions~\ref{prop:weak}, \ref{prop:weakW}, and~\ref{prop:qo} below). 

The Galerkin discretization of \eqref{eq:sl} results in a nonlinear algebraic system of $\dim(\W)$ equations, whose structure depends on the subspace $\W$
as well as on the choice of basis functions for $\W$, that must be solved by an appropriate numerical approximation procedure. The purpose of the present paper is to propose an iterative solver for these nonlinear equations, which is based on \emph{linear} Galerkin discretizations in the individual steps, and to prove its global convergence. In particular, in the case
where $\W$ is realized by a sequence of $hp$-finite element subspaces, for analytic forcing $f$  in $\overline{\Omega}$ in \eqref{eq:sla}, we show that it is possible to terminate the iterative procedure after finitely many steps in a way that the resulting numerical approximations exhibit \emph{exponential rates of convergence} in $\H^1(\Omega)$ as expressed in terms of the overall computational work (see the main result, Theorem~\ref{thm:MainRes}, of this paper).

\subsection*{Previous work}
Earlier results on the numerical analysis of iterative Galerkin methods for nonlinear
elliptic partial differential equations (PDEs), which are closely related to the approach developed here, can be found, e.g., in~\cite{BDMS15,BDMR17}, where low-order (so-called ``$h$-version'') FEM for the solution of  \eqref{eq:sl} have been proposed and analyzed. In the context of $hp$-version FEM, we refer to the article~\cite{CongreveWihler:17} on strongly monotone elliptic problems. More generally, the so-called \emph{iterative linearized Galerkin (ILG)} methodology for nonlinear PDEs, which is based on \emph{a-posteriori} residual estimators for the Galerkin discretization error and the iterative solution of the associated nonlinear algebraic equations, has been developed in the series of papers~\cite{AmreinWihler:15,HeidPraetWihler:21,HeidWihler:20} and the references cited therein. In addition, we refer to the works~\cite{ErnVohralik:13,El-AlaouiErnVohralik:11}, where the use of inexact linear solvers was also taken into account. Finally, for iterative linearized $hp$-adaptive (discontinuous) Galerkin discretizations for semilinear PDEs, we mention the paper~\cite{HoustonWihler:18}.
%
\subsection*{Contribution}
If the source term $f$ in the PDE~\eqref{eq:sl} is analytic in $\overline{\Omega}$,
then the (unique) weak solution of~\eqref{eq:sl} belongs to a class of functions that are analytic at each point 
in (the open domain)~$\Omega$, with quantitative control on the loss of analyticity of these functions
in a vicinity of the corner points; see the recent work~\cite{HS22_1031}. More precisely, the weak solution belongs to a class of functions that satisfy analytic estimates in a scale of corner-weighted Hilbertian Sobolev spaces of Kondrat'ev type.
As it is well-known, see, e.g., \cite{FS20_2675,ChSphp98},
membership of a function in such classes allows for the establishment of \emph{exponential approximability}
in suitable discrete spaces of continuous piecewise polynomial functions in $\Omega$
on regular, simplicial partitions that are geometrically refined 
towards the corners of $\Omega$. 
For the semilinear elliptic boundary value problem \eqref{eq:sl} under consideration,
we formulate the corresponding exponential approximability result in 
Theorem~\ref{thm:ExpAppr} below.

The Galerkin discretization of a suitable variational formulation of~\eqref{eq:sl} amounts to a possibly large nonlinear algebraic system of equations for the 
coefficients in the representation of the Galerkin approximation.
We present an iterative scheme for the efficient numerical solution of 
these equations, and establish geometric 
convergence of the iterates in the $\H^1(\Omega)$-norm in Theorem~\ref{thm:ILGConv} below.
The proof is based on energy type arguments, and thereby, 
constants in the contraction rate bounds do not explicitly 
depend on the dimension $d_N$ of the underlying $hp$-Galerkin approximation space.

The main achievement of our paper is a fully explicit, iterative $hp$-version finite element scheme for the numerical solution of~\eqref{eq:sl}, with a rigorous theoretical analysis of its exponential approximation properties and computational complexity. The key idea is to terminate the iteration in tandem with the  (asymptotically) exponential convergence rates in the $hp$-Galerkin discretization, and thereby, to reach a prescribed target accuracy $0<\eps\ll 1$ in the $\H^1(\Omega)$-norm. We prove that this procedure requires at most $\O(|\log(\eps)|^7)$ floating point operations, which is confirmed computationally in \S\ref{sec:numerics} of this paper, where we report on several numerical experiments for the proposed $hp$-ILG method that are in complete agreement with the theoretical error bounds.

\subsection*{Notation}
For a summability index $r\in (0,\infty]$, 
we denote by $\L{r}$ the usual spaces of Lebesgue measurable, 
$r$-integrable functions in $\Omega$. 
For a nonnegative integer $k$, 
we write $\H^k(\Omega)$ to be the Hilbertian Sobolev space of
$k$-times weakly differentiable functions in $\Omega$ whose weak derivatives
of order $k$ belong to $\L{2}$, with the convention $\H^0(\Omega) = \L{2}$. 
Furthermore, 
in any statements about ``error-vs.-work'', 
the notion of ``work'' refers to an integer number of floating-point operations.
\subsection*{Outline}
In \S\ref{sec:WeakSol} we discuss the weak formulation of the boundary value problem~\eqref{eq:sl}. 
In particular, in \S\ref{sec:AnReg}, we introduce corner-weighted 
Sobolev spaces of Kondrat'ev type, and revisit an analytic regularity
result for the weak solution from~\cite{HS22_1031}.
Then, \S\ref{sec:ILG} deals with the design and analysis of the ILG approach in abstract Galerkin spaces.
Furthermore, in the context of $hp$-FE spaces, we establish the geometric rate of convergence, i.e. 
reaching a target tolerance $\eps>0$ in $\O(|\log(\eps)|)$ many iteration steps. 
We develop an error vs. complexity analysis of the $hp$-discretization, including the nonlinear
iteration with finite termination, in \S\ref{sec:hpFEM} . In \S\ref{sec:numerics}
we present some numerical examples that confirm
the theoretical results within a computational context. Finally,
in \S\ref{sec:Concl} we summarize the principal findings of this
article.
\section{Weak solution}
\label{sec:WeakSol}
We introduce a weak formulation of \eqref{eq:sl}, and establish the
existence and uniqueness of weak solutions. In addition, from \cite{HS22_1031}, we recapitulate the analytic regularity of the weak solution~$u$ of \eqref{eq:sl} in $\Omega$, i.e., we discuss
\emph{a priori} estimates for $u$ in corner-weighted 
Sobolev spaces of arbitrary order, for given forcing $f$ that is
analytic in $\Omega$. 
Crucially, this corner-weighted, analytic regularity is known to imply
exponential approximability of so-called $hp$-FE 
approximations which will be discussed in
\S\ref{sec:ExpConv} below.
\subsection{Weak formulation}
The weak formulation of~\eqref{eq:sl} is to find $\u\in\HGamma$ such that
\begin{equation}\label{eq:weak}
a(\u,v)+\lambda b(u;v)=\int_\Omega fv\,\dx\qquad\forall v\in\HGamma,
\end{equation}
where we define the standard bilinear form via
\begin{equation}\label{eq:a}
a(v,w):=\int_\Omega\nabla v\cdot\nabla w\,\dx,\qquad
v,w\in\HGamma,
\end{equation}
and, for any given $u\in\HGamma$, the linear form
\begin{equation}\label{eq:b}
v\mapsto b(u;v) := \int_\Omega u^{2q+1}v\,\dx,\qquad v\in\HGamma.
\end{equation}

\begin{remark}\label{rem:bbound}
Exploiting the continuous Sobolev embedding bound
\begin{equation}\label{eq:Sobolev}
\|v\|_{\L{2(q+1)}}\le C_q(\Omega)\NN{v}_{\Ltwo}\qquad\forall v\in\HGamma,
\end{equation}
where $C_q(\Omega)>0$ is a constant depending only on $\Omega$, $\cD$ and $\cN$, and $q\in[0,\infty)$,
for any fixed $u\in\HGamma$, we note that the mapping $v\mapsto b(u;v)$ is a bounded linear functional on $\HGamma$. Indeed, this follows immediately from H\"older's inequality, which implies that
\begin{align*}
|b(u;v)|
&\le \|u\|^{2q+1}_{\L{2(q+1)}}\|v\|_{\L{2(q+1)}}
\le C_q(\Omega)^{2(q+1)}\NN{u}^{2q+1}_{\Ltwo}\NN{v}_{\Ltwo},
\end{align*}
for any $v\in\HGamma$.
\end{remark}

The following technical result is instrumental for our 
ensuing convergence analysis of the $hp$-ILG discretization of~\eqref{eq:sl}.

\begin{lemma}
For any $q\in\mathbb{N}_0$ the nonlinear form from~\eqref{eq:b} satisfies the monotonicity property 
\begin{equation}\label{eq:bmonotone}
b(u;u-v)-b(v;u-v)\ge 0\qquad\forall u,v\in\HGamma.
\end{equation}
Furthermore, the (local) Lipschitz continuity bound holds
\begin{multline}\label{eq:bLip}
|b(u;v)-b(w;v)|\\
\le(2q+1)C_q(\Omega)^{2(q+1)}\left(\|\nabla u\|_{\Ltwo}+\|\nabla(u-w)\|_{\Ltwo}\right)^{2q}
 \|\nabla(u-w)\|_{\Ltwo}\|\nabla v\|_{\Ltwo},
\end{multline}
for any $u,v,w\in\HGamma$, with $C_q(\Omega)$ the constant from~\eqref{eq:Sobolev}.
\end{lemma}

\begin{proof}
The estimate~\eqref{eq:bmonotone} follows immediately by noticing that the nonlinear reaction term $g(t):=t^{2q+1}$, $t\in\mathbb{R}$, occurring in~\eqref{eq:sl} satisfies the monotonicity property
\[
(g(t_1)-g(t_2))(t_1-t_2)\ge 0\qquad\forall t_1,t_2\in\R;
\]
then, for any $u,v\in\HGamma$, we infer that
\begin{equation}
\label{eq:monotonicity}
b(u;u-v)-b(v;u-v)
=\int_\Omega\left(u^{2q+1}-v^{2q+1}\right)(u-v)\,\dx\ge 0.
\end{equation}
In order to derive the bound~\eqref{eq:bLip},  for any $x,y\in\R$, we observe the binomial formula
\[
x^{2q+1}=\sum_{k=0}^{2q+1}\binom{2q+1}{k}(x-y)^{2q+1-k}y^k,
\]
or equivalently,
\[
x^{2q+1}-y^{2q+1}
=\sum_{k=0}^{2q}\binom{2q+1}{k}(x-y)^{2q+1-k}y^k.
\]
Therefore, for any $u,v,w\in\HGamma$, it follows that
\[
|b(w;v)-b(u;v)|
\le\sum_{k=0}^{2q}\binom{2q+1}{k}\int_\Omega|u-w|^{2q+1-k}|u|^k|v|\,\dx.
\]
Applying a (triple) H\"older inequality, for $1\le k\le 2q$, we note that
\begin{multline*}
\int_\Omega|u-w|^{2q+1-k}|u|^k|v|\,\dx\\
\le
\left(\int_\Omega|u-w|^{2(q+1)}\right)^{\nicefrac{(2q+1-k)}{2(q+1)}}
\left(\int_\Omega|u|^{2(q+1)}\right)^{\nicefrac{k}{2(q+1)}}
\left(\int_\Omega|v|^{2(q+1)}\right)^{\nicefrac{1}{2(q+1)}};
\end{multline*}
it is straightforward to see that this bound extends to the case $k=0$. Hence, we obtain
\begin{align*}
|b(u;v)-b(w;v)|
&\le \|v\|_{\L{2(q+1)}}\sum_{k=0}^{2q}\binom{2q+1}{k}
\|u-w\|^{2q+1-k}_{\L{2(q+1)}}
\|u\|^k_{\L{2(q+1)}}\\
&= \|u-w\|_{\L{2(q+1)}}\|v\|_{\L{2(q+1)}}\sum_{k=0}^{2q}\binom{2q+1}{k}
\|u-w\|^{2q-k}_{\L{2(q+1)}}
\|u\|^k_{\L{2(q+1)}}.
\end{align*}
Noticing that
\[
\binom{2q+1}{k}
\le
(2q+1)\binom{2q}{k},\qquad 0\le k\le 2q,
\]
and recalling the Sobolev embedding bound~\eqref{eq:Sobolev}, completes the proof.
\end{proof}

\subsection{Well-posedness of weak formulation}
\label{sec:WellPsd}
The semilinear boundary value problem~\eqref{eq:sl} admits a unique weak solution in $\HGamma$.


\begin{proposition}[Existence, uniqueness, and stability]\label{prop:weak}
For the weak formulation~\eqref{eq:weak} there exists exactly one solution $u\in\HGamma$. 
This solution is stable in the sense that
\begin{subequations}
\begin{align}
&\NN{\u}_{\Ltwo}\le\rho,
\label{eq:ustab}
\intertext{where}
&\rho:=C_q(\Omega)\|f\|_{\L{\nicefrac{2(q+1)}{(2q+1)}}},
\label{eq:rho}
\end{align}
\end{subequations}
with $C_q(\Omega)>0$ the Sobolev embedding constant from~\eqref{eq:Sobolev}.
\end{proposition}

\begin{proof}
We aim to apply the main theorem on monotone potential operators. To this end, we define the functional
\[
\F(v):=\frac12 \int_\Omega|\nabla v|^2\,\dx+\frac{\lambda}{2(q+1)}\int_\Omega v^{2(q+1)}\,\dx-\int_\Omega fv\,\dx,\qquad v\in\HGamma,
\]
which is well-defined owing to the Sobolev embedding bound~\eqref{eq:Sobolev}. For any $v\in\HGamma$, employing H\"older's inequality, and exploiting~\eqref{eq:Sobolev}, we note that
\begin{equation}\label{eq:fbound}
\left|\int_\Omega fv\,\dx\right|\le
\|f\|_{\L{\nicefrac{2(q+1)}{(2q+1)}}}\|v\|_{\L{2(q+1)}}
\le C_q(\Omega) \|f\|_{\L{\nicefrac{2(q+1)}{(2q+1)}}}\NN{v}_{\Ltwo}\quad\forall v\in\HGamma.
\end{equation}
Hence, we deduce that
\[
\F(v)\ge\frac12\NN{v}^2_{\Ltwo}
- C_q(\Omega) \|f\|_{\L{\nicefrac{2(q+1)}{(2q+1)}}}\NN{v}_{\Ltwo}\to+\infty,
\]
whenever $\NN{v}_{\Ltwo}\to\infty$, i.e., $\F$ is \emph{weakly coercive}. 
For the G\^ateaux derivative of $\F$ we have that
\[
\dprod{\F'(u)}{v}=\int_\Omega\nabla u\cdot\nabla v\,\dx+\lambda\int_\Omega u^{2q+1}v\,\dx-\int_\Omega fv\,\dx\qquad\forall u,v\in\HGamma,
\]
i.e., the weak formulation~\eqref{eq:weak} is equivalent to 
\begin{equation}\label{eq:F'}
\dprod{\F'(u)}{v}=0\qquad\forall v\in\HGamma,
\end{equation}
where $\dprod{\cdot}{\cdot}$ signifies the dual product.
Furthermore, using~\eqref{eq:bmonotone}, we observe that
\[
\dprod{\F'(u)-\F'(v)}{u-v}
=\NN{(u-v)}^2_{\Ltwo}+\lambda(b(u;u-v)-b(v;u-v))\ge\NN{(u-v)}^2_{\Ltwo}>0
\]
for all $u\neq v$ in $\HGamma$, i.e., $\F'$ is strictly monotone. Then, applying~\cite[Thm.~25.F]{Zeidler:IIB} guarantees that there exists a unique solution of the weak formulation~\eqref{eq:F'}. Finally, for $v=u$ in~\eqref{eq:weak}, applying~\eqref{eq:fbound}, and invoking~\eqref{eq:rho}, we conclude that
\[
\NN{u}^2_{\Ltwo}
\le a(u,u)+\lambda b(u;u)
=\int_\Omega fu\,\dx
\le C_q(\Omega) \|f\|_{\L{\nicefrac{2(q+1)}{(2q+1)}}}\NN{u}_{\Ltwo}
\le \rho\NN{u}_{\Ltwo},
\]
which completes the argument.
\end{proof}

\subsection{Analytic regularity in scales of Hilbertian, corner-weighted Sobolev spaces}
\label{sec:AnReg}
If the underlying domain $\Omega$ for the boundary value problem~\eqref{eq:sl} exhibits corners then
it is well-known that the inverse Dirichlet Laplace operator
does not provide full elliptic regularity. Indeed, for $f\in \Ltwo$,
the solution belongs to $\H^{2}(\Omega^\circ)$ 
in any open, interior subset $\Omega^\circ$, with $\overline{\Omega^\circ}\subset\Omega$. 
The $\H^{2}$-regularity does in general, however, not hold in
a vicinity of the corners. 
To describe this behaviour more precisely, 
the scale of Hilbertian, corner-weighted Sobolev spaces 
$\H^{k,l}_{\mat\beta}(\Omega)$, for $k \geq l \in\mathbb{N}$, 
has been introduced, e.g., in~\cite{BabuskaGuo:88}; 
here, $\mat\beta=(\beta_1,\ldots,\beta_m)$ is a vector of (scalar)
corner weight exponents $\beta_i\in[0,1)$, $i=1,\ldots,m$, 
that are associated to the $m$ corner points $\mat c_1,\ldots,\mat c_m$ of $\Omega$. 
Then, based on the corner-weight function
$
\Phi_{\mat\beta}(\mat x):=\prod_{i=1}^m \mathrm{dist}(\mat x,\mat c_i)^{\beta_i},
$
for integers $k$ and $\ell$ with $ k\geq \ell \ge 0$, 
we introduce the corner-weighted Sobolev norms 
$\| \circ \|_{\H^{k,\ell}_{\mat\beta}}$
in $\Omega$ via
\begin{equation}\label{eq:NrmWghted}
\|v\|_{\H^{k,\ell}_{\mat\beta}(\Omega)}^2
:=
\|v\|^2_{\H^{\ell-1}(\Omega)}
+
\sum_{j=\ell}^k
\sum_{\alpha_1+\alpha_2=j}
\|\Phi_{\mat\beta+j-\ell}\partial^{\alpha_1}_{x_1}\partial^{\alpha_2}_{x_2} v \|^2_{\Ltwo},
\end{equation}
and 
define corner-weighted Sobolev spaces
\[
\H^{k,\ell}_{\mat\beta}(\Omega)
=\left\{
v\in\H^{1}(\Omega):\,\|v\|_{\H^{k,\ell}_{\mat\beta}(\Omega)} < \infty
\right\}.
\]
In \eqref{eq:NrmWghted}, 
for $\ell = 0$, the term $\|v\|^2_{\H^{\ell-1}(\Omega)}$ is dropped. 
Referring to, e.g., \cite{ChSphp98} and the references cited therein, it also holds that
\begin{align*}
\H^{2,2}_{\mat\beta}(\Omega) \subset C(\overline{\Omega})& \quad \text{with continuous embedding},
\intertext{and}
\H^{2,2}_{\mat\beta}(\Omega) \subset \H^1(\Omega) &\quad \text{with compact embedding}.
\end{align*}

Based on the corner-weighted Sobolev spaces $\H^{k,\ell}_{\mat\beta}(\Omega)$ of finite order $k$,
we introduce \emph{corner-weighted, analytic classes}  
$\B^\ell_{\mat\beta}(\Omega)$.
\begin{definition} [Weighted analytic class $\B^\ell_{\mat\beta}(\Omega)$]
Let $\ell \geq 0$ be an integer. 
A function $v$ belongs to the class $\B^\ell_{\mat\beta}(\Omega)$ 
if
\begin{enumerate}[\rm(1)]
\item
$v$ belongs to $\H^{k,\ell}_{\mat\beta}(\Omega)$ for all integer $k\geq \ell$,
and 
\item
it holds that 
there exist constants $C_v, d_v > 0$ 
such that
$$
\forall k\geq \ell:\quad 
\big\| \Phi_{\mat\beta + k -\ell} |\D^k v | \big\|_{\Ltwo} 
\leq 
C_v d_v^{k-\ell} (k-\ell)! 
\;,
$$
with the notation $\big|\D^k v\big| = \sum_{\alpha_1+\alpha_2=k} \left| \partial^{\alpha_1}_{x_1}\partial^{\alpha_2}_{x_2} v \right|$.
\end{enumerate}
\end{definition}
For $1\le i\le m$, we denote by $\omega_i\in(0,2\pi)$ the interior angle of $\Omega$ at the corner $\mat c_i$. The following regularity result was established in~\cite{HS22_1031}. 
\begin{proposition}[Regularity of weak solution]\label{prop:reg}
Let $\mat\beta\in[0,1)^m$ such that, for any $i\in\{1,2,\dots,m\}$, 
it holds
\begin{equation}\label{eq:DN}
\begin{split}
\begin{cases}
\beta_i>1-\nicefrac{\pi}{\omega_i}&\text{if }\{i-1,i\}\subset\cD\text{ or }\{i-1,i\}\subset\cN,\\ 
\beta_i>1-\nicefrac{\pi}{2\omega_i}&\text{otherwise.} 
\end{cases}
\end{split}
\end{equation}
Then for $f\in \B^0_{\mat\beta}(\Omega)$ in~\eqref{eq:sl}, the solution $u\in\HGamma$ of~\eqref{eq:sl} belongs to $\B^2_{\mat\beta}(\Omega)$.
\end{proposition}

\begin{remark}
We note that the first bound in~\eqref{eq:DN} refers to a corner point~$\mat c_i$ 
that connects two adjacent ``Dirichlet edges'' or two ``Neumann edges'', 
while the second bound applies to a corner point $\mat c_i$ where the type of boundary condition changes.
\end{remark}
%
\section{Iterative linearized Galerkin (ILG) discretization} 
\label{sec:ILG}
%

It is well-known that certain Galerkin discretizations of elliptic equations, 
which are based on sequences $\{ \W \}_{N\geq 1}$ of subspaces $\W\subset \HGamma$
of dimension at most $N$, are able to achieve \emph{exponential rates of convergence} 
if the underlying weak solution exhibits $\B^2_{\mat\beta}(\Omega)$-regularity as provided in Proposition~\ref{prop:reg} above; see, e.g., \cite{ChSphp98} and the reference therein. We point out, however, that such results are based on \emph{assuming} that the corresponding nonlinear Galerkin equations are solved exactly. Evidently, this is an unrealistic hypothesis in the context of practical numerical solution methods for \emph{nonlinear} problems; indeed, in the absence of an exact nonlinear solver, the corresponding system of $\O(N)$ nonlinear algebraic equations for the unknowns in the Galerkin solution $u_N\in \W$ need to be solved approximately by some iterative process that must be run to an accuracy of the order of the best approximation error. 
This approach will be addressed in this section.

We begin by providing an abstract error analysis of approximate solutions 
in generic, closed Galerkin subspaces of $\HGamma$.
Subsequently, we consider the specific context of 
$hp$-FE discretizations in~\S\ref{sec:hpFEM}.
\subsection{Finite-dimensional Galerkin approximations}
In the sequel, let $\W \subset\HGamma$ denote any \emph{finite-dimensional} (and thus closed) subspace,
and consider the Galerkin discretization 
of the boundary value problem~\eqref{eq:sl} on $\W$ in weak form: 
find $\U\in \W$ such that
\begin{equation}\label{eq:weakGalerkin}
a(\U,v)+\lambda b(\U;v)=\int_\Omega fv\,\dx\qquad\forall v\in \W.
\end{equation}

The following result follows verbatim as in the proof of~Proposition~\ref{prop:weak}.

\begin{proposition}[Well-posedness of Galerkin discretization]\label{prop:weakW}
There exists a unique solution $\U\in \W$ 
of the discrete weak formulation~\eqref{eq:weakGalerkin}, 
and the stability bound $\NN{\U}_{\Ltwo}\le\rho$ holds, 
with $\rho$ from~\eqref{eq:rho}.
\end{proposition}

%
The Galerkin approximation $\U$ is a quasi-optimal approximation of $u$.
\begin{proposition}[Quasi-optimality]\label{prop:qo}
For the error between the solution $\u\in\HGamma$ of~\eqref{eq:weak} 
and 
its Galerkin approximation $\U\in \W$ of~\eqref{eq:weakGalerkin}, 
the following bound holds
\[
\NN{(\u-\U)}_{\Ltwo}\le C\inf_{w\in \W}\NN{(\u-w)}_{\Ltwo},
\]
for a constant $C>0$ depending only on $\lambda$, $q$, 
$\Omega$, and $f$ in~\eqref{eq:sl}.
\end{proposition}
\begin{proof}
Introducing the error $\errex:=\u-\U$, 
we observe the \emph{Galerkin orthogonality} of $\errex$:
\[
a(\errex,v)+\lambda b(\u;v)-\lambda b(\U;v)=0\qquad\forall v\in \W.
\]
Thereby, for any $w\in \W$, the following identity holds
\begin{align*}
\NN{\errex}^2_{\Ltwo}
&=a(\errex,\u-w)+a(\errex,w-\U)\\
&=a(\errex,\u-w)-\lambda b(\u;w-\U)+\lambda b(\U;w-\U)\\
&=a(\errex,\u-w)-\lambda b(\u;\u-\U)+\lambda b(\U;\u-\U)
-\lambda b(\u;w-\u)+\lambda b(\U;w-\u),
\end{align*}
for the exact Galerkin approximation $\U \in \W$. 
Exploiting the monotonicity property stated in~\eqref{eq:bmonotone}, 
we observe that
\[
\NN{\errex}^2_{\Ltwo}
\le
\NN{\errex}_{\Ltwo}\NN{(\u-w)}_{\Ltwo}
-\lambda b(\u;w-\u)+\lambda b(\U;w-\u).
\]
Recalling the stability estimate \eqref{eq:ustab}, 
which holds for both $\u$ and $\U$, cf.~Proposition~\ref{prop:weakW}, 
and applying~\eqref{eq:bLip}, 
we deduce that
\[
\NN{\errex}_{\Ltwo}
\le
\NN{(\u-w)}_{\Ltwo}
+\lambda
(2q+1)C_q(\Omega)^{2(q+1)}\left(3\rho\right)^{2q}
\|\nabla(\u-w)\|_{\Ltwo},
\]
with $\rho$ from~\eqref{eq:rho}. 
Since $w\in \W$ was chosen arbitrarily, the proof is complete.
\end{proof}
\subsection{Iterative solution}
We consider an iterative linearization 
for the solution of the discrete Galerkin formulation~\eqref{eq:weakGalerkin} 
in terms of a \emph{Picard scheme (also called Zarantanello iteration)}, 
cf.~\cite{HeidWihler:20,CongreveWihler:17}: 
starting from 
any initial guess $\U_0\in\HGamma$, 
for a (fixed) parameter $0<\alpha\le 1$, 
consider a sequence 
$\{\uW_{n}\}_{n\ge 1}\subset \W$ 
that is generated by the iteration
\begin{equation}\label{eq:weakWit}
a(\U_{n+1},v)
=
(1-\alpha)a(\U_n,v)
+\alpha \left(\int_\Omega fv\,\dx-\lambda b(\U_n;v)\right)
\qquad\forall v\in \W,
\end{equation}
for $n\ge 0$.

\begin{proposition}[Contraction of ILG procedure]\label{prop:itconv}
For any initial guess $\U_0\in\HGamma$, the iteration~\eqref{eq:weakWit}, 
with 
\begin{equation}\label{eq:alpha}
\alpha\in(0,1]\cap(0,\nicefrac{2}{(L^2+1)}),
\end{equation}
where
\begin{equation}\label{eq:L}
L:=\lambda
(2q+1)C_q(\Omega)^{2(q+1)}\left(2\rho+\|\nabla\U_0\|_{\Ltwo}\right)^{2q},
\end{equation}
satisfies the bound
\[
\NN{(\U-\U_{n+1})}_{\Ltwo}\le
r^{n+1}_\alpha\|\nabla(\U-\U_0)\|_{\Ltwo}\qquad\forall n\ge 0,
\]
with the contraction constant
\begin{equation}\label{eq:q}
r_\alpha:=
\sqrt{(1-\alpha)^2
+\alpha^2L^2}<1.
\end{equation}
Here, $\U\in \W$ 
is the solution of the Galerkin formulation~\eqref{eq:weakGalerkin},  
$\rho$ is defined in~\eqref{eq:rho}, 
and $C_q(\Omega)$ is the Sobolev embedding constant from~\eqref{eq:Sobolev}.
In particular, 
the iteration~\eqref{eq:weakWit} converges strongly in $\HGamma$, i.e.,
$
\lim_{n\to\infty}\NN{(\U-\U_n)}_{\Ltwo}=0.
$
\end{proposition}

\begin{remark}
For the value $\alpha_\star:=\left(1+L^2\right)^{-1}<1$, with $L$ from~\eqref{eq:L}, 
we note that 
$
r_{\min}:=r_{\alpha_\star}=L(1+L^2)^{-\nicefrac{1}{2}}<1
$
is the minimal value of $r_{\alpha}$ in~\eqref{eq:q}.
\end{remark}

\begin{proof}[Proof of Proposition~\ref{prop:itconv}]
We proceed by induction with respect to $n\ge0$. To this end, 
we follow along the lines of the classical theory for monotone operators,
see, e.g., \cite[\S3.3]{Necas:83}, whereby we work with a \emph{local} Lipschitz continuity property similar to the analysis in~\cite{BeckerBrunnerInnerbergerMelenkPraetorius:23}.
Let $R:=\|\nabla(\U-\U_0)\|_{\Ltwo}$, and suppose that
\begin{equation}\label{eq:ind}
\NN{(\U-\U_k)}_{\Ltwo}\le Rr_\alpha^k,\qquad\text{for }k=0,\ldots,n,
\end{equation}
which is obviously true for $n=0$. 
Then, by virtue of Remark~\ref{rem:bbound}, 
we are able to define a unique (Riesz representer) $\Res_n\in \W$ 
by the weak formulation
\begin{equation}\label{eq:Res}
a(\Res_n,v)=b(\U;v)-b(\U_n;v)\qquad\forall v\in \W.
\end{equation}
In light of~\eqref{eq:bLip} we note that 
\begin{multline*}
|b(\U;v)-b(\U_n;v)|\\
\le
(2q+1)C_q(\Omega)^{2(q+1)}\left(\NN{\U}_{\Ltwo}+\NN{(\U-\U_n)}_{\Ltwo}\right)^{2q}
 \NN{(\U-\U_n)}_{\Ltwo}\NN{v}_{\Ltwo},
\end{multline*}
for any $v\in\HGamma$. 
Furthermore, by virtue of Proposition~\ref{prop:weakW} 
and due to the induction assumption~\eqref{eq:ind}, 
it follows that
\begin{equation}\label{eq:ind2}
|b(\U;v)-b(\U_n;v)|
\le
(2q+1)C_q(\Omega)^{2(q+1)}\left(\rho+R\right)^{2q}
 \NN{(\U-\U_n)}_{\Ltwo}\NN{v}_{\Ltwo},
\end{equation}
for any $v\in\HGamma$, i.e., the right-hand side of~\eqref{eq:Res} 
is a bounded linear functional with respect to~$v$, 
and thus, $\Res_n\in \W$ is well-defined. 
In addition, letting $v=\Res_n$ in~\eqref{eq:Res} and~\eqref{eq:ind2}, 
we infer that
\begin{equation}\label{eq:Resbound}
\NN{\Res_n}_{\Ltwo}\le(2q+1)C_q(\Omega)^{2(q+1)}\left(\rho+R\right)^{2q}
 \NN{(\U-\U_n)}_{\Ltwo}.
\end{equation}
Moreover, 
upon subtracting the weak formulations \eqref{eq:weakGalerkin} and~\eqref{eq:weakWit}, 
we observe that
\begin{align*}
a(\U-\U_{n+1},v)
&=(1-\alpha)a(\U-\U_n,v)
+\alpha\lambda \left(b(\U_n;v)-b(\U;v)\right)\\
&=(1-\alpha)a(\U-\U_n,v)
-\alpha\lambda a(\Res_n,v),
\end{align*}
for any $v\in \W$, which implies the identity
\[
\U-\U_{n+1}
=(1-\alpha)(\U-\U_n)-\alpha\lambda\Res_n.
\]
Thus,
\begin{multline*}
\NN{(\U-\U_{n+1})}^2_{\Ltwo}\\
=(1-\alpha)^2\NN{(\U-\U_{n})}^2_{\Ltwo}
-2\alpha(1-\alpha)\lambda a(\Res_n,\U-\U_n)
+\alpha^2\lambda^2\NN{\Res_n}^2_{\Ltwo}.
\end{multline*}
Inserting $v=\U-\U_n$ in~\eqref{eq:Res}, and using~\eqref{eq:bmonotone}, we notice that
$
a(\Res_n,\U-\U_n)\ge 0.
$
Therefore, recalling~\eqref{eq:Resbound}, we deduce that
\begin{align*}
\NN{(\U-\U_{n+1})}^2_{\Ltwo}
\le\left((1-\alpha)^2
+\alpha^2\lambda^2
(2q+1)^2C_q(\Omega)^{4(q+1)}\left(\rho+R\right)^{4q}\right)
 \NN{(\U-\U_n)}^2_{\Ltwo}.
\end{align*}
Employing Proposition~\ref{prop:weakW}, gives
\[
R\le\NN{\U}_{\Ltwo}+\NN{\U_0}_{\Ltwo}\le \rho+\NN{\U_0}_{\Ltwo},
\]
and thereby the bound $\NN{(\U-\U_{n+1})}_{\Ltwo}\le Rr_\alpha^{n+1}$ holds,
which concludes the induction argument.
\end{proof}
\subsection{Exponential convergence of the ILG iteration} 
\label{sec:ExpConv}
We consider sequences $\{\W\}_{N\ge 1}\subset\HGamma$ of 
subspaces $\W$ of finite dimension $\dim(\W)\sim N$ 
which allow for exponentially convergent approximations of functions 
$u\in \B^2_{\mat\beta'}(\Omega)$, where we set $\mat\beta':=(\beta'_1,\dots,\beta'_m)\in[0,1)^m$ with
\begin{equation}
\label{eq:wbeta}
\beta'_i:=\frac12\left(1+\max(\beta_i,1-\nicefrac{\pi}{\omega_i})\right),\qquad 1\le i\le m.
\end{equation}
More precisely, 
we suppose that there exist constants $K,b,\gamma>0$ independent of $N$ 
with
\begin{equation}\label{eq:exp}
\inf_{w\in\W}\NN{(u-w)}_{\Ltwo}
\le 
K\exp\left(-bd_N^{\gamma}\right), 
\qquad\text{where} \;\;
d_N \geq \dim(\W),\; N\ge 0
\;.
\end{equation}
Specific choices of $\{\W \}_{N\geq 1}$ that realize \eqref{eq:exp} with $\gamma = \nicefrac13$ 
will be given in \S\ref{sec:hpFEM} below.
\begin{theorem}[Convergence of Picard iteration]
\label{thm:ILGConv}
Suppose that the right-hand side in~\eqref{eq:sl} satisfies $f\in\B^0_{\mat\beta}(\Omega)$. 
Furthermore, consider a family of Galerkin spaces~$\{\W\}_{N\ge 1}$ 
that features the exponential approximation property~\eqref{eq:exp}. 
Then, for any fixed $N\ge 1$ and any initial guess $\U_0\in \HGamma$, 
performing $n=\mathcal{O}\left(d_N^\gamma\right)$ steps 
of the Picard iteration scheme~\eqref{eq:weakWit}, 
with suitable $\alpha\in(0,1]$ cf.~Proposition~\ref{prop:itconv}, 
leads to the error estimate
\begin{equation}\label{eq:hpILGErr}
\NN{(\u-\U_n)}_{\Ltwo}\le C\exp\left(-\kappa d_N^{\gamma}\right),
\end{equation}
where $\u\in \HGamma$ is the weak solution of~\eqref{eq:sl}, 
with constants $C,\kappa>0$ that are independent of $N$, 
and with $\gamma>0$ from~\eqref{eq:exp}.
\end{theorem}

\begin{proof}
If $\mat\beta'\in(0,1)^m$ is defined as in~\eqref{eq:wbeta} 
then $f\in \B^0_{\mat\beta'}(\Omega)$. 
Owing to Proposition~\ref{prop:reg}, 
we note that the solution of~\eqref{eq:weak} satisfies~$u\in\B^2_{\mat\beta'}(\Omega)$. 
Then, for any subspace $\W$ and any $n\ge 0$, using the triangle inequality, 
we have
\[
\NN{(\u-\U_n)}_{\Ltwo}
\le
\NN{(\u-\U)}_{\Ltwo}
+
\NN{(\U-\U_n)}_{\Ltwo},
\]
where $\U\in\W$ is the solution of~\eqref{eq:weakGalerkin}. 
The first term on the right-hand side of the above inequality 
can be estimated by combining the quasi-optimality property from Proposition~\ref{prop:qo} 
and the exponential approximation bound~\eqref{eq:exp}, 
thereby yielding
\begin{equation} \label{eq:exp2}
\NN{(\u-\U)}_{\Ltwo}\le C_1\exp\left(-bd_N^{\gamma}\right),
\end{equation}
for a constant $C_1>0$. 
Furthermore, the second term is bounded by employing 
Proposition~\ref{prop:itconv}, i.e., we have that
\begin{equation}\label{eq:aux1}
\NN{(\U-\U_n)}_{\Ltwo}\le C_2\, r_\alpha^n,
\end{equation}
with constants $0 < r_\alpha < 1$ and $C_2>0$. 
Choosing $n=\mathcal{O}\left(d_N^\gamma\right)$ shows 
that 
\begin{equation}\label{eq:aux2}
r_\alpha^n \le \exp\left(-b' d_N^{\gamma}\right),
\end{equation}
for some $b'>0$ independent of $n$ and $N$,
which completes the proof.
\end{proof}
\section{Exponential Convergence and $\eps$-Complexity of $hp$-ILGFEM}
\label{sec:hpFEM}
%
Based on the convergence result for the Picard iteration~\eqref{eq:weakWit} from
Theorem~\ref{thm:ILGConv}, we are ready to formulate 
a fully discrete ILG scheme, and to establish its 
exponential convergence, respectively, polylogarithmic $\eps$-complexity 
for the numerical approximation of the semilinear boundary value problem \eqref{eq:sl} 
in a polygon $\Omega$ with analytic data $f\in \B^0_\beta(\Omega)$.
We first discuss 
the design of (sequences of) so-called $hp$-FE spaces $\W$,
which, together with the analytic regularity $u\in \B^2_{\mat\beta'}(\Omega)$ 
stated in Proposition~\ref{prop:reg}, affords the exponential convergence rate
\eqref{eq:exp} with $\gamma = \nicefrac13$ for the corresponding exact Galerkin solutions
$U \in \W$, cf.~\eqref{eq:weakGalerkin}.
We then address the realization of the iteration \eqref{eq:weakWit},
in particular, the solver complexity of the linear part based on the
$hp$-FE spaces $\W$.  
Here, we emphasize that the Laplace operator
and the monomial nonlinearity in \eqref{eq:sl} allow for an \emph{exact computation}
(assuming absence of rounding errors)
of all bilinear and nonlinear forms (in terms of $a(\cdot,\cdot)$ and $b(\cdot;\cdot)$, respectively) 
in the ILG iteration \eqref{eq:weakWit},
and thus, to dispense with a quadrature error analysis. 
%
\subsection{$hp$-Approximations on geometric corner meshes}
\label{sec:hpGeoCorMesh}
Under the analytic $\B^2_{\mat\beta}(\Omega)$-regularity  
in scales of corner-weighted Sobolev spaces 
of the weak solution $u\in\HGamma$ of \eqref{eq:sl}, 
which has been shown in \cite{HS22_1031}, see also \S\ref{sec:AnReg},
subspace sequences $\{ \W \}_{N\geq1}$ in \eqref{eq:weakGalerkin} 
that satisfy~\eqref{eq:exp} with $\gamma = \nicefrac13$ 
have been
constructed in \cite{FS20_2675,ChSphp98} and the references cited therein. 
We briefly recapitulate their construction here. 

For integers $k\geq 1$, we denote by $\cT_k$ 
a regular, simplicial partition of the polygon $\Omega$ 
into at most $\O(k)$ open triangles $T$. 
We assume that the partitions $\{ \cT_k \}_{k\geq 1}$ 
are nested, uniformly shape regular, 
and 
\emph{geometrically refined} to the corners 
$\set{C}=\{\mat c_i\}_{i=1}^m\subset\partial\Omega$ of the polygon~$\Omega$; 
the last property refers to the existence of constants 
$C>0$ and  $\sigma \in (0,1)$
such that it holds
\begin{subequations}\label{eq:Geo}
\begin{align}
\forall k \;\;\forall T\in\cT_k: \overline{T} \cap {\set C} = \emptyset : &
\qquad  
0  < \sigma \leq \nicefrac{\diam(T)}{\dist(T, \set C)} \leq \sigma^{-1} 
\;,\label{eq:Geo1}
\intertext{and}
\forall k \;\; \forall T\in \cT_k: \overline{T}\cap {\set C} \ne \emptyset: & 
\qquad
\diam(T) \leq C \sigma^k
\;.\label{eq:Geo2}
\end{align}
\end{subequations}
Geometric partitions satisfying the above conditions can be constructed 
in any polygon $\Omega$ by recursive bisection refinement 
departing from a regular, admissible initial triangulation $\cT_1$ of~$\Omega$.
We shall refer to triangulations generated in this way as 
\emph{geometric corner meshes} in~$\Omega$.
The aforementioned recursive bisection refinement constructions yield 
the existence of a constant $C>0$ (depending on~$\Omega$ and on $\cT_1$) 
such that
$$
\forall k\in \N:\quad \#(\cT_k) \leq C k \;.
$$

With the geometric corner meshes in hand,
the spaces $\W$ of dimension $d_N=\O(N)$
consist of continuous functions in $\overline{\Omega}$ whose restriction to each element $T\in \cT_k$, $k\geq 1$, 
are polynomials of uniform total polynomial degree $p_k\geq 1$, 
and obey the homogenous Dirichlet boundary condition~\eqref{eq:Dirichlet}. 
The spaces $\W$ are built 
on the sequences $\{ \cT_k \}_{k\geq 1}$ of geometric corner meshes
\begin{equation}\label{eq:WNhpFEM}
\W = 
\Spk{\cT_k}
= 
\{ v \in \HGamma \mid \forall T \in \cT_k : v|_{T} \in \mathbb{P}^{p_k}(T) \}
\;,
\end{equation}
with $\mathbb{P}^{p_k}(T)$ signifying the space of all polynomials of
total degree at most $p_k$ over $T$, $T\in\cT_k$, $k\geq1$, 
whereby we choose $p_k \simeq c k$, for some fixed sufficiently large constant $c>0$.
Then it follows that there exists a constant $C>0$ such that 
$$ 
d_N = \dim(\Spk{\cT_k}) \leq C k^3 \simeq p^3_k,
$$
for all $N$.
Based on the above construction~\eqref{eq:WNhpFEM} of the $hp$-FE spaces~$\W$, 
the corresponding Galerkin discretizations of~\eqref{eq:sl}, cf.~\eqref{eq:weakGalerkin}, 
exhibit the exponential approximability \eqref{eq:exp} with $\gamma = \nicefrac13$.
\begin{theorem}[Exponential convergence of $hp$-Galerkin approximations]
\label{thm:ExpAppr}
For $\mat\beta\in[0,1)^m$, consider the semilinear boundary value problem \eqref{eq:sl} with data $f\in \B^0_{\mat\beta}(\Omega)$. 
Then, for the Galerkin projection of the weak solution $u\in \HGamma$ of \eqref{eq:sl} 
in the $hp$-FE spaces $\W$ of dimension $d_N=\O(N)$ from~\eqref{eq:WNhpFEM},
there exists constants $b,C>0$ (generally depending on $u$) 
such that the following exponential approximability bound holds
\begin{equation}\label{eq:ExpAppr}
\inf_{w_N \in \W} \NN{(u - w_N)} _{\Ltwo}
\leq 
C\exp(-b N^{\nicefrac13}) 
\;.
\end{equation}
In particular the exact Galerkin solutions $U \in \W$ 
in \eqref{eq:weakGalerkin},
whose existence and uniqueness was ensured in Proposition~\ref{prop:weakW},
satisfy the exponential convergence bound \eqref{eq:exp} with $\gamma =\nicefrac13$.
\end{theorem}
\begin{proof}
Based on the assumed analytic regularity $f\in \B^0_{\mat\beta}(\Omega)$ 
on the data, by Proposition~\ref{prop:reg} and \cite{HS22_1031},
 the unique weak solution $u\in \HGamma$ of \eqref{eq:sl}
belongs to the corner-weighted analytic class $\B^2_{\mat\beta}(\Omega)$.
The exponential approximability \eqref{eq:ExpAppr} is then a direct consequence
of well-known approximation properties of the $hp$-FE spaces $\W$ introduced above, see
e.g. \cite{FS20_2675,ChSphp98} and the references cited therein.
\end{proof}
\begin{remark}
For subspaces $\W\subset \HGamma$ 
spanned by continuous, piecewise polynomial functions 
on regular, simplicial partitions $\cT_k$, $k\geq 1$, 
of the polygon $\Omega$,
\emph{under the present assumption of the polynomial nonlinearity in the reaction term in~\eqref{eq:sl}, we emphasize that 
the stiffness matrix corresponding to the Galerkin discretization of the bilinear form~$a(\cdot,\cdot)$ from~\eqref{eq:a},
      and the discrete representation of 
the nonlinear form~$b(\cdot;\cdot)$ from~\eqref{eq:b} can be evaluated exactly}, e.g., by 
Gaussian quadrature rules of sufficiently high order 
applied on each element $T\in \cT_k$, $k\geq 1$.
\end{remark}
\subsection{Solution of the nonlinear algebraic system}
\label{sec:NonlinSol}
On any given finite-dimensional subspace $\W\subset\HGamma$, 
we emphasize that the solution of the \emph{linear} algebraic system 
resulting from the ILG scheme~\eqref{eq:weakWit} features the \emph{same} stiffness matrix (associated to the Laplace operator) in each iteration step. 
On a related note, for $n\ge 0$, exploiting Remark~\ref{rem:bbound}, 
we can introduce a unique (Riesz representative)~$\res_n\in \W$ by
\begin{equation}\label{eq:ItAlg}
a(\res_n,v)= \int_\Omega fv\,\dx -\lambda b(\U_n;v) \qquad\forall v\in\W,
\end{equation}
which allows us to write the iterative scheme~\eqref{eq:weakWit} in the strong form
\begin{equation}\label{eq:weakit}
\U_{n+1}=(1-\alpha)\U_n + \alpha\res_n,\qquad n\ge 0.
\end{equation}

We will now estimate the operation count incurred in the implementation of iteration
\eqref{eq:weakit}. 
We begin our discussion with the observation that, on a given Galerkin space~$\W$, 
the Cholesky decomposition of the stiffness matrix corresponding to the bilinear form $a(\cdot,\cdot):\, \W\times \W\to \R$ from~\eqref{eq:a} 
can be performed once and for all, so that the computational work is limited to one backsolve only in each step of \eqref{eq:weakit}.
Without any further assumption, the Cholesky
decomposition of the matrix corresponding to 
the bilinear form $a(\cdot,\cdot)$ incurs work of $\O(d_N^3)$,  
where $d_N=\dim(\W)=\O(N)$.
From Theorem~\ref{thm:ILGConv},
the number $n$ of iterations of \eqref{eq:weakit} required to reach the approximation error bound
\eqref{eq:exp2} is given by~$\O(d_N^\gamma)$.
\subsubsection{Evaluation cost of the nonlinear term}
\label{sec:EvalNonlin}

In the sequel, for ease of notation, for meshes $\cT_k$ with an associated polynomial degree~$p_k$, cf.~\S\ref{sec:hpGeoCorMesh}, we will simply write $\cT_p$ and $p$, instead, respectively.
Recalling that $\Omega\subset \R^2$ is a polygon, 
and that $\W$ in~\eqref{eq:WNhpFEM} 
is a space of continuous, piecewise polynomial functions
on a regular triangulation $\cT_p$ of $\Omega$  with a corresponding
(uniform throughout $\cT_p$) total degree $p\geq 1$, 
we require $\O(p^2)$ many Gauss points in each triangle $T\in \cT_p$ for an \emph{exact integration} of the (polynomial) nonlinear form $\int_T w^{2q+1}v\,\dx$, for $w,v\in\mathbb{P}^p(T)$; here, our focus is on high-order $hp$-approximations where $p\gg q$, and thus 
\[
\deg(w^{2q+1}v)\le 2p(q+1)=\O(p).
\]
Therefore, in total, the evaluation of the nonlinear form
$$
b(w;v) = \sum_{T\in \cT_p} b_T(w;v) \;:= \sum_{T\in \cT_p} \int_T w^{2q+1}v \,\dx \;,\quad v,w\in \W, 
$$
occurring on the right-hand side of~\eqref{eq:ItAlg}, 
amounts to a work of $\O(\#({\cT_p}) p^2)$. 
In addition, if the right-hand side source function $f$ in~\eqref{eq:sl} 
is a polynomial (which we henceforth assume) then we note that the evaluation of the linear form $\int_\Omega fv\,\dx$ 
amounts to the same cost.
%
\subsubsection{Error vs. work for dense Cholesky decomposition}
Assuming that we have at hand the Cholesky factors of the global stiffness matrix 
associated to the bilinear form~$a(\cdot,\cdot)$, 
which requires $\O(N^3)$ operations \emph{once only}, 
each iteration step in~\eqref{eq:ItAlg} necessitates 
only one backsolve of $\O(N^2)$ work as well as one evaluation of the nonlinearity.
Then, 
in view of our considerations in~\S\ref{sec:EvalNonlin}, 
the total work for the approximate solution of the nonlinear algebraic 
system using $n$ iteration steps amounts to a computational work of
\[
\O(N^3) + n (\O(N^2) + \O(\#(\cT_p) p^2)).
\]
In case of $hp$-FE approximations with 
$\#(\cT_p) = \O(p)$ 
many triangles and 
$d_N = \O(N) = \O(p^3)$, $\gamma=\nicefrac13$, 
cf.~Theorem~\ref{thm:ExpAppr}, 
this transforms to $\O(p^9)+ n(\O(p^6) + \O(p^3))$. 
Therefore, in light of Theorem~\ref{thm:ILGConv}, 
applying $n=\O(d^\gamma_N)=\O(N^{\nicefrac13})=\O(p)$ iterations in~\eqref{eq:ItAlg} 
to obtain consistency with the Galerkin discretization error, 
a total work of
\begin{equation}\label{eq:Op}
\O(p^9)+ p(\O(p^6) + \O(p^3)) = \O(p^9) + \O(p^7)\;
\end{equation}
many operations are required. 
In particular, 
if we do not exploit any sparsity structure in the $hp$-FE matrices
of the linear principal part of~\eqref{eq:sl}, 
the total work for the computation of the approximate
Galerkin solution to the accuracy of the $hp$-FE discretization error (with exact
solution of the Galerkin equations) is still dominated by the work for the 
linear solve $\O(N^3)$.
\subsubsection{Static condensation Cholesky decomposition} 
\label{sec:static_condensation}
We discuss how to reduce the above operation count 
by exploiting a suitable separation 
of the polynomial basis functions on each
triangle $T\in \cT_p$, $p\geq 1$, 
into \emph{$3$ nodal, $\O(p)$ face and $\O(p^2)$ internal modes}; 
see \cite{ChSphp98,SzIMBPFEM,SBJSch06,FKDN2015} for details.
Using static condensation on each of the $\O(p)$ elements in the 
geometric mesh $\cT_p$, $p\geq 1$, and noting that 
\[
\nicefrac{\#\text{dof}}{\#\text{elements}} 
=\nicefrac{\O(p^3)}{\O(p)}
= \O(p^2), 
\]
costs $\O(p^6)$ flops per element $T\in\cT_p$, $p\geq 1$.
Hence, multiplying by the number of elements, results in the following result.
\begin{proposition}\label{prop:WorkChol}
The work for a global Cholesky decomposition of the stiffness matrix 
(corresponding to the bilinear form $a(\cdot,\cdot)$ from~\eqref{eq:a}) 
on the $hp$-FE space $\Sp{\cT_p}$, $p\ge1$, 
with prior elementwise static condensation 
requires asymptotically, as $p\to\infty$, $\O(p^7)$ flops.
\end{proposition}
\begin{proof}
We apply elementwise static condensation which yields only $\O(p)$ ``external'' unknowns on $\O(p)$ 
inter-element edges in $\cT_p$. Consequently, upon condensation, a
global dense linear system of size $\O(p^2)$ is obtained. 
The work of a Cholesky decomposition 
for this condensed system scales as $\O((p^2)^3) = \O(p^6)$ flops. 
This, in turn, implies that work for a nested dissection version of the initial Cholesky decomposition of 
the bilinear form $a(\cdot,\cdot): \W\times \W\to \R$ 
can be reduced from $\O(d_N^3) = \O(p^9)$ to $\O(p^7)$ flops.
\end{proof}

\begin{remark}
In particular, the above Proposition~\ref{prop:WorkChol} shows that the work for the initial Cholesky decomposition of the  condensed stiffness matrix of the linear part 
$a(\cdot,\cdot):\W\times \W\to \R$ is asymptotically (as $d_N \to \infty$) of the same order 
as the work required to run the iteration~\eqref{eq:ItAlg} on 
\emph{each} Galerkin space $\W=\Sp{\cT_p}$ in $n(\O(N^2) + \O(\#(\cT_p) p^2))=\O(p^7)$ operations, cf.~\eqref{eq:Op},
to termination at ``$hp$ approximation accuracy''~\eqref{eq:ExpAppr}.
\end{remark}

\begin{remark} [Static condensation Cholesky decomposition exploiting affine equivalence] 
\label{rmk:StatCholAffin}
The argument in Proposition~\ref{prop:WorkChol} is not sharp in the sense that it 
does not exploit the constant-coefficient structure of the linear principal part of the operator
in \eqref{eq:sl}. Indeed, in this case, in the element stiffness matrix
generation there is no necessity to employ numerical integration. 
In contrast,
for a linear principal part of the form $-\Div (\mat A(\mat x) \nabla(\cdot))$ with 
\emph{non-constant coefficients}, i.e., for a diffusion matrix $\mat A:\Omega \to \R^{2\times 2}_{\mathrm{sym}}$ with variable coefficient functions $a_{ij}$ that are analytic in $\overline\Omega$, numerical quadrature in the 
element stiffness matrices is necessary. 
The work needed to build the global stiffness matrix then scales by a constant factor 
compared to the work required to compute the stiffness matrix in the constant coefficient case. 
The use of quadrature will cause 
a further consistency error that is, however, of the same order as the discretization error bounds. 
We therefore expect the present error vs. work bound to hold also in this more general setting.

In addition, assuming that $\cT_p$, $p\geq 1$, 
is a regular partition of $\Omega$ into affine-equivalent
triangles, all element stiffness matrices are affine equivalent to the reference
element stiffness matrix. 
Therefore, element stiffness matrix Cholesky factorization
and Schur complement formation could be done once and for all on the reference element.
This would reduce the cost bound of nested dissection Cholesky  factorization of 
the stiffness matrix for the linear principal part to $\O(p^6)$. 
We refer to \cite{SBJSch06} for further exploitation of reference element stiffness and
mass matrix sparsity afforded by the choice of a particular set of shape functions.
As the total work for the evaluation of the nonlinearity during the ILG iteration
is already $\O(p^7)$, we do not account for this lower complexity in the present article.
\end{remark}
%
\subsection{Exponential convergence of fully discrete iterative linearized $hp$-FE approximation}
We will now investigate the complexity of the fully discrete iteration~\eqref{eq:weakWit}. 
To this end, in order to deal with potential quadrature errors of integrals 
involving the right-hand side function~$f$ in~\eqref{eq:sla}, 
we make the following assumption:

\begin{apt}\label{Q}
For a sequence $\{ \cT_p\}_{p\geq 1}$ of regular partitions of $\Omega$ into triangles, with geometric corner refinement and associated polynomial degrees $p=1,2,\ldots$, cf.~\S\ref{sec:hpGeoCorMesh}--\ref{sec:NonlinSol}, we suppose that the right-hand side function $f\in\Ltwo\cap\B^0_{\mat\beta}(\Omega)$ in~\eqref{eq:sla} can be approximated by a corresponding sequence of functions $\Pf$ with $\Pf|_{T}\in\mathbb{P}^{\tp}(T)$ for all $T\in\cT_p$, $p=1,2,\ldots$, where $\tp=\O(p)$, such that
\begin{equation}\label{eq:fexp}
\|f-\Pf\|_{\Ltwo}\le C\exp(-cN^{\nicefrac13}),
\end{equation}
with $N=N(p)=\dim(\Sp{\cT_{p}}) = \O(p^3)$, as in~\eqref{eq:ExpAppr}, and constants $c,C>0$.
\end{apt}

\begin{remark}
This assumption is satisfied, evidently, if $f\in\Ltwo\cap\B^0_{\mat\beta}(\Omega)$ 
is 
a polynomial function in $\Omega$, and,
in particular,
for all functions that are real-analytic in~$\overline{\Omega}$.
More general conditions and fully discrete approximation methods, 
which are based on point-queries of $\mat A(\mat x)$, 
cf.~Remark~\ref{rmk:StatCholAffin}, and $f(\mat x)$ in $\Omega$, 
will be studied in a forthcoming article.
\end{remark}

Instead of the discrete iteration~\eqref{eq:weakWit}, 
which relies on exact integration of $f$, 
we now consider the \emph{fully discrete $hp$-ILG scheme} 
given by
\begin{equation}
\label{eq:weakWit'}
a(\tU_{n+1},v)
=
(1-\alpha)a(\tU_n,v)
-\alpha \lambda b(\tU_n;v)
+\alpha \sum_{T\in\cT_p}Q_{T,\tp}(\Pf v)\qquad \forall v\in\W,
\end{equation}
where $Q_{T,\tp}$ is an elementwise quadrature rule on each triangle $T\in\cT_p$, 
which is based on $\O(p^2)$ (e.g. Gauss-type) quadrature points, 
and integrates any polynomial in~$\mathbb{P}^{\tp+p}(T)$ \emph{exactly}. 
Similar to our previous discussion in~\S\ref{sec:EvalNonlin}, 
we note that the work for the evaluation of the quadrature in~\eqref{eq:weakWit'} amounts to $\O(\#(\cT_p)p^2)$.

\begin{lemma}\label{lem:aux}
Let $\alpha$ satisfy~\eqref{eq:alpha}, and suppose that the iterations~\eqref{eq:weakWit} and~\eqref{eq:weakWit'} are initiated with the same starting guess~$U_0\in\WW_N$. Then, under Assumption~\ref{Q}, for the difference between the corresponding iterative solutions $U_{n}$ and $\tU_{n}$ the following bound holds
\[
\NN{(\U_{n}-\tU_{n})}_{\Ltwo}\le C\exp\left(-\kappa N^{\nicefrac13}\right),\qquad n\ge 0,
\]
with constants $C,\kappa>0$ independent of $N$.
\end{lemma}

\begin{proof}
We consider the auxiliary problem of finding $\widetilde U\in\WW_N$ such that
\[
a(\widetilde U,v)+\lambda b(\widetilde U;v)=\sum_{T\in\cT_p}Q_{T,\tp}(\Pf v)\qquad \forall v\in \W.
\]
By Assumption~\ref{Q}, we observe that
\[
a(\widetilde U,v)+\lambda b(\widetilde U;v)=\int_\Omega \Pf v\,\dx\qquad \forall v\in \W.
\]
This shows that Proposition~\ref{prop:itconv} can be applied to the perturbed iteration~\eqref{eq:weakWit'}, with the \emph{same} contraction constant $r_\alpha$ from~\eqref{eq:q}. In particular, we have the bound
\[
\NN{(\widetilde U-\tU_{n+1})}_{\Ltwo}\le
r^{n+1}_\alpha\|\nabla(\widetilde U-\U_0)\|_{\Ltwo}\qquad\forall n\ge 0.
\]
In addition, 
for the difference of $\widetilde U$ and the solution~$U$ of~\eqref{eq:weakGalerkin} it holds
\[
a(U-\widetilde U,v)+\lambda \left(b(U;v)-b(\widetilde U;v)\right)=\int_\Omega (f-\Pf) v\,\dx\qquad \forall v\in \W.
\]
Testing with $v=U-\widetilde U\in\WW_N$, and applying~\eqref{eq:monotonicity}, leads to
\[
\NN{(U-\widetilde U)}^2_{\Ltwo}\le \|f-\Pf\|_{\Ltwo}\|U-\widetilde U\|_{\Ltwo},
\]
which, by means of the Poincar\'e inequality, results in
\[
\NN{(U-\widetilde U)}_{\Ltwo}\le C_P\|f-\Pf\|_{\Ltwo},
\]
for a constant $C_P>0$ only depending on~$\Omega$. Hence, owing to the triangle inequality, we deduce that
\begin{align*}
\NN{(U-\tU_{n+1})}_{\Ltwo}
&\le\NN{(U-\widetilde U)}_{\Ltwo}+\NN{(\widetilde U-\tU_{n+1})}_{\Ltwo}\\
&\le C_P\|f-\Pf\|_{\Ltwo}+r^{n+1}_\alpha\|\nabla(\widetilde U-\U_0)\|_{\Ltwo}.
\end{align*}
The first term on the right-hand side of the above bound can be estimated with~\eqref{eq:fexp}, whilst the second term can be dealt with as in~\eqref{eq:aux1}--\eqref{eq:aux2} with $d_N=\O(N)$ and $\gamma=\nicefrac13$, see~\S\ref{sec:hpGeoCorMesh}.
\end{proof}

We are now prepared to state and prove our main result.
\begin{theorem}[{Polylogarithmic $\eps$-complexity}]
\label{thm:MainRes}
Suppose that the semilinear boundary value problem~\eqref{eq:sl} features a corner-weighted
analytic source term $f\in \Ltwo \cap \B^0_{\mat\beta}(\Omega)$ for which Assumption~\ref{Q} can be satisfied.
%
Consider the $hp$-ILG discretization based on a sequence $\{ \cT_{p}\}_{p\geq 1}$
of regular partitions of $\Omega$ into triangles, with geometric corner refinement,
cf.~\eqref{eq:Geo}, 
and on Galerkin projections onto the associated discrete 
spaces $\WW_{N(p)}=\Sp{\cT_{p}}$, $p\ge 1$, from \eqref{eq:WNhpFEM}, 
with $N = N(p) = \dim(\Sp{\cT_{p}}) = \O(p^3)$,
and with $n(p) = \O(p)$ steps of the linearized iteration~\eqref{eq:weakWit} on $\Wp$. 
Then, the following hold true:
\begin{enumerate}[\rm(a)]

\item For every $0<\eps\leq 1$, there exists a polynomial degree
$
p = \O(|\log(\eps)|)
$
such that upon applying
$n(p)$ steps of the fully discrete $hp$-ILG iteration procedure \eqref{eq:weakWit'} 
the approximate Galerkin solutions $\tU_{n(p)} \equiv \tU_{n} \in \Wp$ 
satisfy the error bound
$
\NN{(u - \tU_{n(p)})}_{\Ltwo} \leq \eps \;.
$

\item The total computational cost $\work(\eps) \in \N$, 
measured in terms of float point operations necessary to
compute the $hp$-ILG approximations
$\tU_{n(p)} \in  \Wp$ to accuracy $\eps>0$, 
is bounded by 
\begin{equation}\label{eq:EpsErr}
\work(\eps) \leq C(1+|\log(\eps)|)^{7} \;.
\end{equation}
\item In terms of the number $N(p)$ of unknowns, 
or 
in terms of $\work(\eps)$ required for the computation of the 
$hp$-ILG discretization of \eqref{eq:sl}, 
there are constants $b,C>0$ such that it holds 
$$
\NN{( u - \tU_{n(p)})}_{\Ltwo}
\leq 
C\exp(-b(N(p))^{\nicefrac13}) 
\leq 
C\exp(-b\work(\eps)^{\nicefrac17}) ,
$$
for $p=1,2,\ldots$
\end{enumerate}
\end{theorem}
\begin{proof}
In accordance with~\eqref{eq:hpILGErr} and Lemma~\ref{lem:aux}, the error $\NN{(u-\tU_{n(p)})}_{\Ltwo}$ for the fully discrete $hp$-ILG scheme~\eqref{eq:weakWit'} after employing $n(p)=\O(p)$ iteration steps is of order $C\exp(-bp)$ for suitable constants $b,C>0$. In particular, the $hp$-ILG iteration error is of the order of the $hp$-discretization error, cf.~\eqref{eq:ExpAppr}. Furthermore, coupling the parameter $p$ (i.e., the polynomial degree,
viz. the number of geometric mesh layers) to the prescribed target accuracy
$0<\eps\leq 1$ as in the statement of the theorem via $p\geq C|\log(\eps)|$
for a sufficiently large uniform constant $C>0$, shows~(a). The work estimate in (b) follows from 
Proposition~\ref{prop:WorkChol} 
and from the discussion in \S\ref{sec:EvalNonlin} on the complexity for 
one evaluation of the $hp$-discretization of the nonlinearity. 
Finally, (c) is a direct consequence of~(a) and~(b).
\end{proof}
\begin{remark}
We remark that the error vs. work bound \eqref{eq:EpsErr} is analogous in asymptotic
order to the (exponential) error vs. work bound for the $hp$-FE Galerkin 
solution of the corresponding linear problem in $\Omega$.
\end{remark}


%
\section{Numerical Experiments} \label{sec:numerics}
In this section we present a series of numerical experiments 
to computationally verify the theoretical error vs. work bounds derived in Theorem~\ref{thm:MainRes} 
for the $hp$-ILG discretization of the semilinear boundary value problem \eqref{eq:sl}. 
To this end, we consider the semilinear boundary value problem~\eqref{eq:sl}, 
for $\lambda=q=1$ and the constant function $f\equiv 1$, 
on two different domains:
\begin{itemize}
	\item {\bf Example 1:} We let $\Omega$ be the unit square $(0,1)^2$.
	\item {\bf Example 2:} Here, $\Omega$ is the ``L-shaped domain'',
        $(-1,1)^2\setminus [0,1)\times(-1,0]\subset {\mathbb R}^2$.
\end{itemize}
\begin{figure}[t!]
\begin{center}
\includegraphics[scale=0.5]{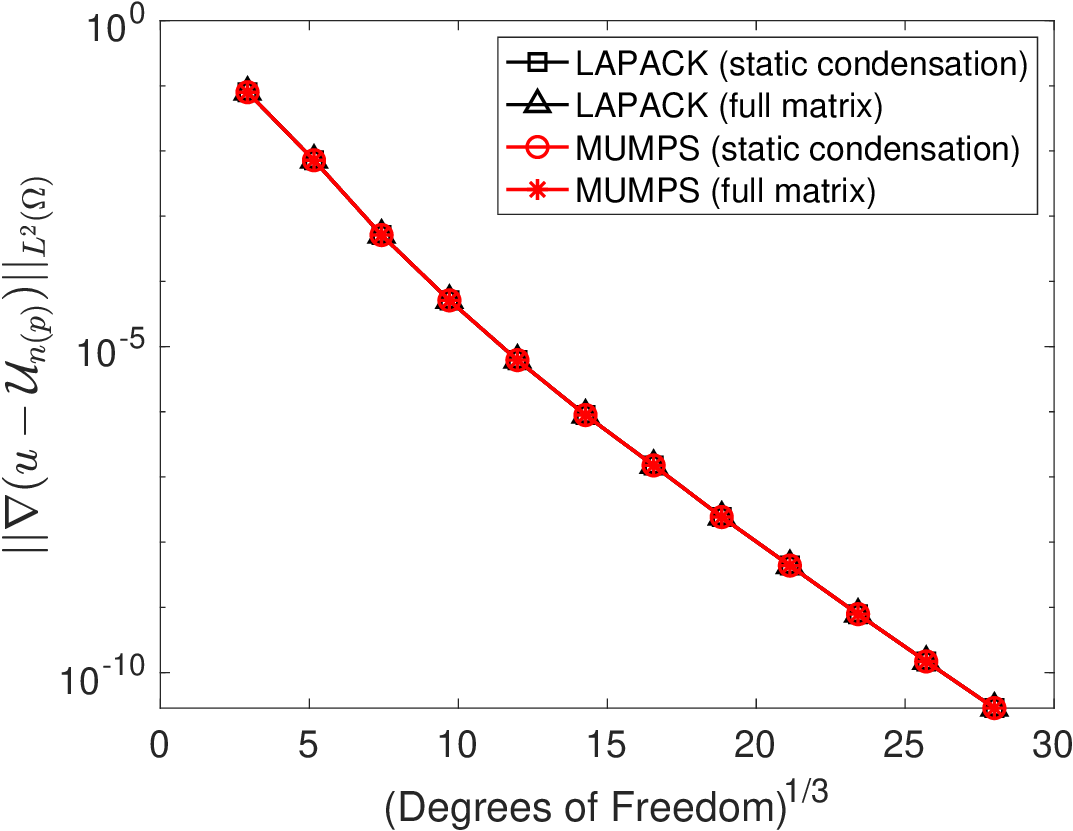} \\
(a) \\
~~\\
\includegraphics[scale=0.5]{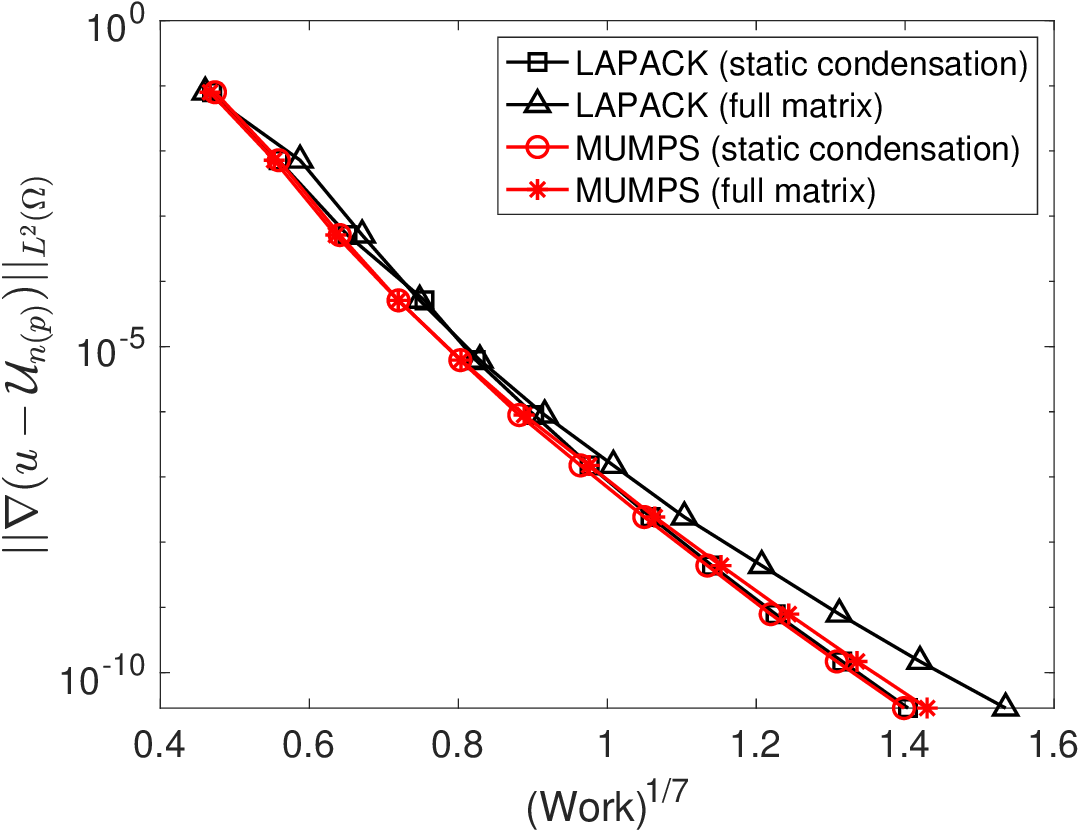} \\
(b)
\end{center}
\caption{Example 1 (square domain): (a) Comparison of the error with respect to the third root of the
number of degrees of freedom; (b) Comparison of the error with respect to the seventh root of the
work (CPU time) measured in seconds.}
\label{fig:results_ex1}
\end{figure}

We remark that in both settings, the analytical solution to \eqref{eq:sl} is unknown and hence a suitably fine $hp$-mesh approximation is computed for the purposes of evaluating the corresponding $L^2(\Omega)$-norm of the (gradient of the) error in the $hp$-ILG approximation defined by \eqref{eq:weakWit'}. Given that $f=1$ throughout this section, the integral involving $f$ can be computed exactly, and hence \eqref{eq:weakWit} and \eqref{eq:weakWit'} are equivalent. The sequence of FE spaces $\Wp := \Sp{\cT_{p}}$, $p\geq 1$, are constructed as outlined in \S\ref{sec:hpGeoCorMesh}; namely, given an initial uniform triangular mesh $\cT_1$, the space $\mathbb{S}^{1}_{\cD}(\Omega;\cT_{1})$ is constructed based on employing (uniform order) continuous piecewise linear polynomials. Subsequently we undertake geometric refinement of each mesh $\cT_p$, $p\geq 1$, whereby elements in the vicinity of the four, respectively, six corners of the domain $\Omega$ are subdivided, while simultaneously (uniformly) increasing the polynomial order $p$. 
For the construction of the underlying $hp$-FE polynomial spaces, we exploit the hierarchical basis defined in \cite{Solin_book_2004}.
Additionally, the sequences of corner-refined, geometric meshes used in the numerical experiments
are built based on employing the standard red-green refinement strategy whereby (temporary) 
green refinement is undertaken to remove hanging nodes in the underlying computational mesh. 
The resulting sequence of meshes are generally not nested; 
however, we observe in the ensuing numerical experiments
that nestedness of partitions is not necessary 
to ensure that the theoretical bounds in Theorem~\ref{thm:MainRes} still hold. 

\begin{figure}[t!]
\begin{center}
\includegraphics[scale=0.5]{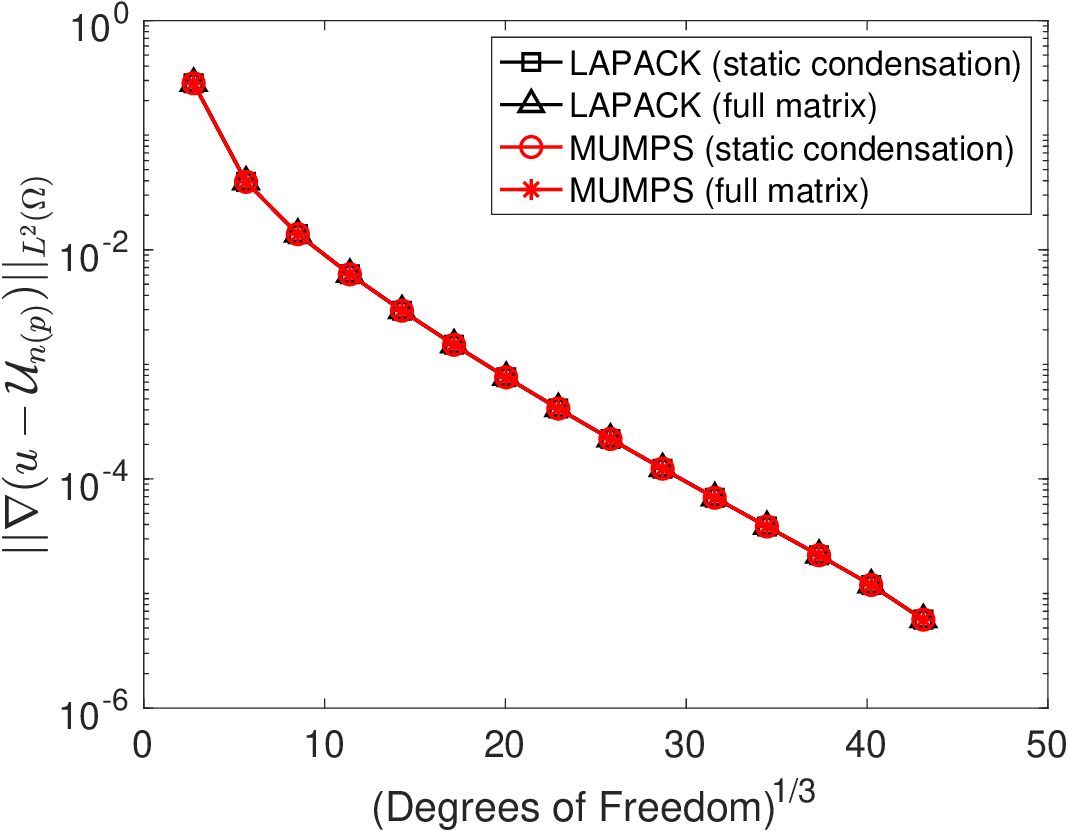} \\
(a) \\
~~\\
\includegraphics[scale=0.5]{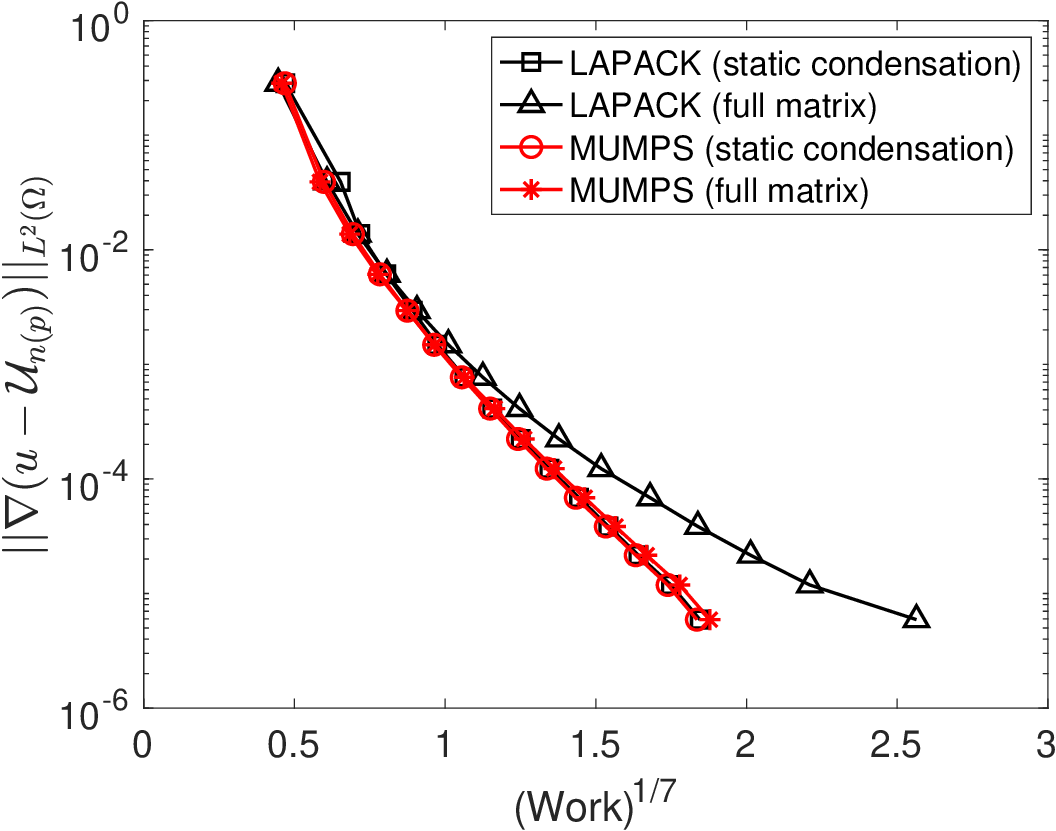} \\
(b)
\end{center}
\caption{Example 2 (L-shaped domain): (a) Comparison of the error with respect to the third root of the
number of degrees of freedom; (b) Comparison of the error with respect to the seventh root of the
work (CPU time) measured in seconds.}
\label{fig:results_ex2}
\end{figure}

For the purposes of comparison, we consider the implementation of two different linear solvers for the factorization of the underlying matrix arising from the bilinear form $a(\cdot,\cdot)$ from~\eqref{eq:a}. In addition to the Cholesky factorization {\tt DPOTRF} implemented in LAPACK for dense matrices, cf. \cite{lapack99}, we also consider the application of the MUltifrontal Massively Parallel Solver (MUMPS), see \cite{MUMPS:2,MUMPS:1,MUMPS:3}, which utilizes sparse (symmetric) storage of the underlying matrix problem. In Figures~\ref{fig:results_ex1}~\&~\ref{fig:results_ex2} we present a comparison of the norm of the error in the computed numerical solution $\tU_{n(p)}$ for $p=1,2,\ldots,12$, respectively, $p=1,2,\ldots,15$, based on employing a starting mesh $\cT_1$ comprising of 32, respectively 24, uniform triangular elements. Once the $hp$-ILG solution $\tU_{n(p)}\in \Sp{\cT_{p}}$ has been computed on a given mesh $\cT_p$, for a particular polynomial degree $p$, this solution is then projected onto the refined finite element space $\mathbb{S}^{p+1}_{\cD}(\Omega;\cT_{p+1})$ to serve as the initial guess in the Picard fixed point scheme \eqref{eq:weakWit'}; for $p=1$ the initial guess is the zero solution. The iteration defined in \eqref{eq:weakWit'} is terminated once the $\ell_2$-norm of the difference in the coefficient vector between two subsequent computed solutions has been reduced by a factor of $10^{-2}$, relative to
the corresponding quantity computed between the initial guess and the first iterate computed using \eqref{eq:weakWit'}; furthermore we set $\alpha = \nicefrac{1}{2}$. 
Here we consider the impact of static condensation,  
as per \S\ref{sec:static_condensation}, 
as well as when both MUMPS and LAPACK are applied to the full matrix problem arising in \eqref{eq:weakWit'}.
In Figures~\ref{fig:results_ex1}(a)~\&~\ref{fig:results_ex2}(a) we plot the $\Ltwo$-norm of the gradient of the error $u-\tU_{n(p)}$
against the third root of the number of degrees of freedom in the FE space $\Wp \equiv \Sp{\cT_{p}}$ on a semi-log plot; 
here we observe that for both examples, as $p$ is increased (and as the mesh is concurrently geometrically refined towards the corners of $\Omega$), 
the convergence lines become straight, 
thereby indicating exponential convergence, cf. Theorem~\ref{thm:MainRes}. 
Of course, all four numerical approaches produce identical results, as we would expect. 
Figures~\ref{fig:results_ex1}(b)~\&~\ref{fig:results_ex2}(b) present analogous results for both examples where we now plot $\|\nabla (u-\tU_{n(p)})\|_{\Ltwo}$ 
against the seventh root of the work, measured in CPU seconds, 
required to compute the $hp$-ILG solution $\tU_{n(p)}$ on each FE space. 
As expected, when static condensation is employed both MUMPS and LAPACK lead to exponential rates of convergence, 
which is in agreement with Theorem~\ref{thm:MainRes}. 
Furthermore, 
when LAPACK is exploited {\em without} static condensation, 
then we clearly observe a degeneration in the performance of the proposed approach. 
In contrast, we do not observe a similar degradation in the performance of the MUMPS solver applied to the full matrix
(i.e., the assembled global stiffness matrix, without static condensation); we postulate that this is due to the fact that MUMPS can automatically exploit the structure of the matrix
within the initial analyse phase, which is performed prior to computation of the factorisation. 
This ensures that exponential rates of convergence, 
\emph{in terms of the work required to compute the $hp$-ILG solution}, 
are retained in this setting.
Finally, we observe that the MUMPS solver, 
which stores the matrix in sparse format, 
is typically slightly more efficient than LAPACK which utilizes dense matrix storage.
\section{Conclusions}
\label{sec:Concl}
For a model semilinear elliptic boundary value problem in a polygonal domain $\Omega\subset \R^2$, 
with a monotone, polynomial nonlinearity and analytic in $\overline{\Omega}$ source term $f$ in \eqref{eq:sl},  
we proved that the underlying solution $u\in\HGamma$ can be approximated numerically 
to accuracy $\eps>0$ with no more than $\O((1+|\log(\eps)|^7)$ many operations.
This is, up to the constant hidden in $\O(\cdot)$, the same asymptotic 
complexity than resulting for the corresponding linear elliptic problem,
if a direct solver is used for the associated (linear) Galerkin system. 
Key in these results is analytic regularity of the weak solution \cite{HS22_1031}
in scales of corner-weighted Sobolev spaces.
The generalization to diffusion matrices with analytic in $\overline{\Omega}$ 
coefficients is possible at the expense of a quadrature error analysis 
in the (linear) principal part of \eqref{eq:sl} by means of an auxiliary result of Strang type.
By similar arguments the presence of nonlinear reactions with variable analytic coefficients 
may be realized as well (not considered here). 
It could also be of interest to combine the present results with 
$hp$-adaptive refinement, similar to \cite{BeckerBrunnerInnerbergerMelenkPraetorius:23}, where 
fixed order, Langrangian FE discretizations were studied.
\bibliographystyle{amsplain}
\bibliography{references} 
\end{document}